\documentclass[11pt]{article}

\usepackage[T1]{fontenc}
\usepackage[utf8]{inputenc}
\usepackage{amsmath,amssymb,amsthm,mathtools}
\usepackage{libertinus-type1}
\usepackage{libertinust1math}
\usepackage[margin=1in]{geometry}
\usepackage{xcolor}
\usepackage{enumitem}
\usepackage{microtype}
\usepackage[colorlinks=true,linkcolor=blue!50!black,citecolor=blue!50!black,urlcolor=blue!50!black]{hyperref}
\hypersetup{
  pdftitle={A universal discriminant formula for pencils of quadrics},
  pdfauthor={Ari Krishna},
  pdfsubject={Algebraic geometry},
  pdfkeywords={pencils of quadrics, discriminants, binary forms, Grassmannians, Chow classes}
}

\usepackage{aliascnt}
\newtheorem{theorem}{Theorem}[section]
\newaliascnt{proposition}{theorem}
\newtheorem{proposition}[proposition]{Proposition}
\aliascntresetthe{proposition}
\newaliascnt{corollary}{theorem}
\newtheorem{corollary}[corollary]{Corollary}
\aliascntresetthe{corollary}
\newaliascnt{lemma}{theorem}
\newtheorem{lemma}[lemma]{Lemma}
\aliascntresetthe{lemma}
\theoremstyle{definition}
\newaliascnt{definition}{theorem}

\aliascntresetthe{definition}
\newaliascnt{remark}{theorem}
\newtheorem{remark}[remark]{Remark}
\aliascntresetthe{remark}
\newaliascnt{example}{theorem}
\newtheorem{example}[example]{Example}
\aliascntresetthe{example}
\newaliascnt{convention}{theorem}
\newtheorem{convention}[convention]{Convention}
\aliascntresetthe{convention}
\usepackage[nameinlink,noabbrev]{cleveref}
\crefname{theorem}{theorem}{theorems}
\Crefname{theorem}{Theorem}{Theorems}
\crefname{proposition}{proposition}{propositions}
\Crefname{proposition}{Proposition}{Propositions}
\crefname{corollary}{corollary}{corollaries}
\Crefname{corollary}{Corollary}{Corollaries}
\crefname{lemma}{lemma}{lemmas}
\Crefname{lemma}{Lemma}{Lemmas}
\crefname{definition}{definition}{definitions}
\Crefname{definition}{Definition}{Definitions}
\crefname{remark}{remark}{remarks}
\Crefname{remark}{Remark}{Remarks}
\crefname{example}{example}{examples}
\Crefname{example}{Example}{Examples}
\crefname{convention}{convention}{conventions}
\Crefname{convention}{Convention}{Conventions}

\newcommand{\bbk}{\mathbb k}
\newcommand{\bbP}{\mathbb P}

\newcommand{\bbZ}{\mathbb Z}

\newcommand{\OO}{\mathcal O}
\newcommand{\cB}{\mathcal B}
\newcommand{\cD}{\mathcal D}
\newcommand{\cE}{\mathcal E}
\newcommand{\cJ}{\mathcal J}
\newcommand{\cL}{\mathcal L}

\newcommand{\cU}{\mathcal U}
\newcommand{\Gr}{\operatorname{Gr}}
\newcommand{\Sym}{\operatorname{Sym}}
\newcommand{\Hom}{\operatorname{Hom}}

\newcommand{\PGL}{\operatorname{PGL}}
\newcommand{\GL}{\operatorname{GL}}
\newcommand{\Conf}{\operatorname{Conf}}
\newcommand{\disc}{\operatorname{Disc}}

\newcommand{\blfootnote}[1]{%
  \begingroup
  \renewcommand{\thefootnote}{}%
  \footnotetext{#1}%
  \addtocounter{footnote}{-1}%
  \endgroup
}

\title{A universal discriminant formula for pencils of quadrics}

\author{Ari Krishna}

\begin{document}

\maketitle


\begin{abstract}
Let $V$ be a vector space of dimension $n+1$ over an algebraically closed field $\bbk$ of characteristic zero, and let
\[
  G_n=\Gr\bigl(2,\Sym^2V^\vee\bigr)
\]
be the Grassmannian parametrizing pencils of quadrics in $\bbP(V)\cong \bbP^n$. The determinant of the universal pencil defines a universal binary form of degree $n+1$. We prove that the divisor $\cD_n\subseteq G_n$ of pencils whose determinant binary form has a multiple root has Chow class
\[
  [\cD_n]=n(n+1)\sigma_1\in A^1(G_n),
\]
where $\sigma_1=c_1(S^\vee)$ and $S$ is the tautological rank-two subbundle on $G_n$. More generally, the higher-contact loci of determinant binary forms are computed by a universal jet formula. We also formulate the determinant-root collision strata as refined pullbacks of the universal collision strata for binary forms. For $n=3$, the main formula recovers the class $12\sigma_1$ for the boundary divisor in $\Gr(2,10)$ that the author established in a prior paper.
\end{abstract}

\section{Introduction}

Let $V$ be a vector space of dimension $n+1$ over an algebraically closed field $\bbk$ of characteristic zero. A pencil of quadrics in $\bbP(V)\cong \bbP^n$ is a two-dimensional subspace
\[
  L\subseteq \Sym^2V^\vee.
\]
After choosing a basis $Q_0,Q_1$ of $L$, the singular members of the pencil are detected by the vanishing of the determinant
\[
  \Delta_L(s,t)=\det(sQ_0+tQ_1).
\]
This is a binary form of degree $n+1$, and its roots record the singular quadrics in the pencil, counted with multiplicity.

Set
\[
  G_n:=\Gr\bigl(2,\Sym^2V^\vee\bigr).
\]
Let $S\subset \Sym^2V^\vee\otimes\OO_{G_n}$ be the tautological rank-two subbundle, and write
\[
  \sigma_1=c_1(S^\vee).
\]
The determinant restricts to the universal pencil as a universal binary $(n+1)$-ic. The question that we answer in this note is: what is the class of the locus where this binary form has a repeated root?

To this end, our main result is the following. 
\begin{theorem}\label{thm:main}
Let $V$ be a vector space of dimension $n+1$ over an algebraically closed field of characteristic zero. Let
\[
  G_n=\Gr\bigl(2,\Sym^2V^\vee\bigr),
\]
and let $S$ be its tautological rank-two subbundle. Let $\cD_n\subset G_n$ be the locus of pencils $L$ for which the determinant binary form $\det(q)|_L$ has a multiple root. Then $\cD_n$ is an effective Cartier divisor and
\[
  [\cD_n]=n(n+1)c_1(S^\vee)=n(n+1)\sigma_1
  \quad\text{in } A^1(G_n).
\]
Equivalently, if $d=n+1$, then
\[
  [\cD_n]=d(d-1)c_1(S^\vee).
\]
\end{theorem}

The proof is a direct consequence of the coordinate-free transformation rule for the discriminant of a binary form. The determinant gives a section
\[
  \Delta\in H^0\left(G_n,\Sym^dS^\vee\otimes M\right),
  \qquad d=n+1,
\]
where
\[
  M=(\det V^\vee)^2
\]
is a fixed one-dimensional vector space. If $F$ is a binary form of degree $d$ with values in a line bundle $\cL$, then its discriminant is a section of
\[
  \cL^{\otimes(2d-2)}\otimes (\det E^\vee)^{\otimes d(d-1)},
\]
where $E$ is the underlying rank-two bundle of variables. Applying this to $E=S$ and $\cL=M$, and observing that $M$ is constant over $G_n$, we get
\[
  [\cD_n]=d(d-1)c_1(S^\vee).
\]

There is also a jet-theoretic form of the same computation. Let
\[
  \pi:\bbP(S)\to G_n
\]
be the universal parameter line of the pencil, and put
\[
  h=c_1\bigl(\OO_{\bbP(S)}(1)\bigr).
\]
Then, the refined class of the locus where the determinant binary $d$-ic has a marked root of multiplicity at least $m$ is
\[
  \pi_*
  \prod_{i=0}^{m-1}
  \left((d-2i)h+i\sigma_1\right).
\]
For $m=2$, this gives
\[
  d(d-1)\sigma_1,
\]
recovering \cref{thm:main}.

The motivation comes from orbit geometry. In the case $n=3$, the Grassmannian
\[
  \Gr(2,\Sym^2\bbk^{4,\vee})\cong \Gr(2,10)
\]
parametrizes pencils of quadrics in $\bbP^3$. The determinant is a binary quartic. On the open locus where this quartic is squarefree, the projective equivalence class of its unordered four roots is measured by the $j$-invariant. In \cite{Krishna26}, this construction yields a one-parameter family of codimension-one $\PGL_4$-orbit closures, and the divisor class $12\sigma_1$ appears as the class of every divisorial fiber of the rational $j$-map.

For all $n$, determinant-root geometry for pencils of quadrics is pulled back from the geometry of binary $(n+1)$-ics. The universal discriminant formula is the first expression of this principle.

This should be distinguished from the finer Segre-symbol stratification of pencils of quadrics. The determinant records the multiplicities of the singular members of the pencil, but it does \textit{not} always record the rank type of those singular quadrics. Thus, the universal discriminant theorem controls the determinant-collision part of the story; separating all Segre strata requires additional incidence calculations along the lower-rank loci of symmetric matrices. Refer to \cite{FMS21} for the Segre-symbol perspective on pencils of quadrics.

\section{The universal determinant binary form}

Let
\[
  \dim V=d=n+1,
  \qquad
  W=\Sym^2V^\vee,
  \qquad
  G_n=\Gr(2,W).
\]
A point of $W$ is a quadratic form on $V$, equivalently, a symmetric linear map
\[
  q:V\to V^\vee.
\]
Taking determinants gives
\[
  \det(q)\in \Hom(\det V,\det V^\vee)\cong (\det V^\vee)^2.
\]
Thus, coordinate-free, the determinant is a homogeneous polynomial map
\[
  \det:W\to M,
  \qquad
  M:=(\det V^\vee)^2,
\]
of degree $d$. Equivalently,
\[
  \det\in \Sym^dW^\vee\otimes M.
\]

Let
\[
  S\subseteq W\otimes\OO_{G_n}
\]
be the tautological subbundle. Restricting the determinant polynomial to $S$ gives a section
\[
  \Delta\in H^0\left(G_n,\Sym^dS^\vee\otimes M\right).
\]
For a point $L\in G_n$, the fiber
\[
  \Delta_L\in \Sym^dL^\vee\otimes M
\]
is the determinant binary form of the pencil $L$.

If $L=\langle Q_0,Q_1\rangle$ and $[s:t]$ are homogeneous coordinates on $\bbP(L)$, then
\[
  \Delta_L(s,t)=\det(sQ_0+tQ_1).
\]
The zero divisor of $\Delta_L$ on $\bbP(L)$ is the divisor of singular quadrics in the pencil. We define
\[
  \cD_n:=\{L\in G_n:\disc(\Delta_L)=0\}.
\]
Equivalently, $\cD_n$ is the locus where the determinant binary form is not squarefree. This convention includes the locus of pencils for which the determinant is identically zero.

\begin{lemma}\label{lem:nonzero}
The discriminant section $\disc(\Delta)$ is not identically zero on $G_n$.
\end{lemma}

\begin{proof}
Choose pairwise distinct scalars
\[
  \lambda_1,\dots,\lambda_d\in \bbk.
\]
Let
\[
  Q_0=\sum_{i=1}^d \lambda_i x_i^2,
  \qquad
  Q_1=\sum_{i=1}^d x_i^2.
\]
Then,
\[
  \det(sQ_0+tQ_1)=\prod_{i=1}^d(s\lambda_i+t).
\]
This binary form has $d$ distinct roots, so its discriminant is nonzero. Therefore, $\disc(\Delta)$ is not the zero section.
\end{proof}

Since $G_n$ is smooth and integral, \cref{lem:nonzero} implies that the zero scheme of $\disc(\Delta)$ is an effective Cartier divisor. Consequently, it remains to compute the line bundle wherein the discriminant section lives.

\section{The discriminant line bundle}

We recall the coordinate-free discriminant of a family of binary forms.

\begin{lemma}\label{lem:disc-line}
Let $B$ be a scheme, let $E$ be a rank-two vector bundle on $B$, and let $\cL$ be a line bundle. Let
\[
  F\in H^0\left(B,\Sym^dE^\vee\otimes \cL\right)
\]
be a family of binary forms of degree $d$. Then the discriminant of $F$ is naturally a section
\[
  \disc(F)
  \in
  H^0\left(
  B,
  \cL^{\otimes(2d-2)}\otimes(\det E^\vee)^{\otimes d(d-1)}
  \right).
\]
\end{lemma}

\begin{proof}
The assertion is local on $B$. After passing to a splitting cover, write
\[
  F=c\prod_{i=1}^d \ell_i,
\]
where each $\ell_i$ is a local section of $E^\vee$ and $c$ is the leading scalar in the chosen trivialization of $\cL$. The discriminant is
\[
  \disc(F)=c^{2d-2}\prod_{i<j}(\ell_i\wedge \ell_j)^2.
\]
Each factor $\ell_i\wedge\ell_j$ is a section of $\det E^\vee$. There are $\binom d2$ unordered pairs, and each is squared, so the determinant factor is
\[
  (\det E^\vee)^{\otimes d(d-1)}.
\]
The discriminant is homogeneous of degree $2d-2$ in the coefficients of $F$, giving the factor $\cL^{\otimes(2d-2)}$.

The expression is symmetric in the roots. Hence, it descends from the splitting cover, and it defines a global section of the displayed line bundle.
\end{proof}

We turn to the proof of our main result. 

\begin{proof}[Proof of \cref{thm:main}]
Apply \cref{lem:disc-line} with
\[
  B=G_n,
  \qquad
  E=S,
  \qquad
  \cL=M=(\det V^\vee)^2,
  \qquad
  F=\Delta.
\]
Then
\[
  \disc(\Delta)
  \in
  H^0\left(
  G_n,
  M^{\otimes(2d-2)}\otimes(\det S^\vee)^{\otimes d(d-1)}
  \right).
\]
The line $M$ is a fixed one-dimensional vector space. It is constant over $G_n$, so
\[
  c_1\bigl(M^{\otimes(2d-2)}\bigr)=0.
\]
Therefore, the divisor class of $\cD_n$ is
\[
  [\cD_n]
  =c_1\left((\det S^\vee)^{\otimes d(d-1)}\right)
  =d(d-1)c_1(S^\vee).
\]
Since $d=n+1$ and $\sigma_1=c_1(S^\vee)$, this is
\[
  [\cD_n]=n(n+1)\sigma_1.
\]
The divisor is effective Cartier by \cref{lem:nonzero}. This proves the theorem.
\end{proof}

\begin{remark}\label{rem:weight}
Equivalently, the discriminant of a scalar binary form of degree $d$ is a relative invariant of $\GL_2$-weight $d(d-1)$. The formula above is the corresponding Chern class calculation on the rank-two bundle $S$.
\end{remark}

\section{A jet formula for higher contact}

We now give a second proof of the divisor formula, and record the higher-contact calculation.

\begin{convention}\label{conv:projective-bundle}
For a vector bundle $E$, we write $\bbP(E)$ for the bundle of one-dimensional subspaces of $E$. Its tautological line subbundle is $\OO_{\bbP(E)}(-1)\subset \pi^*E$, and
\[
  h=c_1\bigl(\OO_{\bbP(E)}(1)\bigr).
\]
For $\pi:\bbP(S)\to G_n$, this convention gives
\[
  \pi_*(h)=1,
  \qquad
  \pi_*(h^2)=c_1(S^\vee),
  \qquad
  \pi_*(h^3)=c_1(S^\vee)^2-c_2(S^\vee).
\]
\end{convention}

Let
\[
  \pi:\bbP(S)\to G_n
\]
be the universal parameter line. Set
\[
  h=c_1\bigl(\OO_{\bbP(S)}(1)\bigr),
  \qquad
  a=c_1(S^\vee).
\]
The universal determinant section induces a section of
\[
  \OO_{\bbP(S)}(d)\otimes M.
\]
The relative cotangent bundle of $\pi$ has first Chern class
\[
  c_1(\Omega_\pi)=-2h+\pi^*a.
\]
Henceforth, we suppress $\pi^*$ from the notation.

Let
\[
  \cL_d:=\OO_{\bbP(S)}(d)\otimes M.
\]
A marked point of the parameter line is a root of multiplicity at least $m$ precisely when the first $m$ relative jets of the determinant section vanish. Let
\[
  \cJ^{m-1}_\pi(\cL_d)
\]
be the relative jet bundle of order $m-1$. It has a filtration whose associated graded pieces are
\[
  \cL_d,
  \quad
  \cL_d\otimes\Omega_\pi,
  \quad
  \dots,
  \quad
  \cL_d\otimes\Omega_\pi^{\otimes(m-1)}.
\]
Thus
\[
  c_m\bigl(\cJ^{m-1}_\pi(\cL_d)\bigr)
  =
  \prod_{i=0}^{m-1}
  \left(dh+i(-2h+a)\right).
\]

\begin{proposition}\label{prop:jet-formula}
Let $d=n+1$. The refined class of the marked multiplicity-$\geq m$ locus for the determinant binary $d$-ic is
\[
  \prod_{i=0}^{m-1}
  \left((d-2i)h+i\sigma_1\right)
  \quad\text{on } \bbP(S).
\]
Its pushforward to $G_n$ is
\[
  \Theta_m
  :=
  \pi_*
  \prod_{i=0}^{m-1}
  \left((d-2i)h+i\sigma_1\right).
\]
When the marked incidence has the expected codimension and is generically one-to-one over its image, $\Theta_m$ is the ordinary fundamental class of the corresponding unmarked higher-contact locus.
\end{proposition}

\begin{proof}
The determinant section on $\bbP(S)$ gives a section of the relative jet bundle
\[
  \cJ^{m-1}_\pi(\cL_d).
\]
The zero scheme of this jet section is the marked multiplicity-$\geq m$ locus. Its refined top Chern class is
\[
  c_m\bigl(\cJ^{m-1}_\pi(\cL_d)\bigr).
\]
Using the filtration above and the identity $c_1(\Omega_\pi)=-2h+a$, this top Chern class is
\[
  \prod_{i=0}^{m-1}\left(dh+i(-2h+a)\right)
  =
  \prod_{i=0}^{m-1}\left((d-2i)h+ia\right).
\]
Then, pushing forward along $\pi$ gives the stated class on $G_n$.
\end{proof}

\begin{corollary}\label{cor:m2}
For $m=2$,
\[
  \Theta_2=d(d-1)\sigma_1.
\]
\end{corollary}

\begin{proof}
For $m=2$, \cref{prop:jet-formula} gives
\[
  \Theta_2
  =
  \pi_*\left(dh\bigl((d-2)h+a\bigr)\right).
\]
Using
\[
  \pi_*(h)=1,
  \qquad
  \pi_*(h^2)=a,
\]
we obtain
\[
\begin{aligned}
  \Theta_2
  &=d(d-2)\pi_*(h^2)+da\,\pi_*(h) \\
  &=d(d-2)a+da \\
  &=d(d-1)a.
\end{aligned}
\]
Since $a=\sigma_1$, this proves the claim.
\end{proof}

We now provide an example. 

\begin{example}[Triple-root locus]\label{ex:triple}
For $m=3$, the formula gives
\[
  \Theta_3
  =
  \pi_*\left[
  dh\bigl((d-2)h+a\bigr)\bigl((d-4)h+2a\bigr)
  \right].
\]
Expanding, we get
\[
\begin{aligned}
  &dh\bigl((d-2)h+a\bigr)\bigl((d-4)h+2a\bigr) \\
  &\qquad =d(d-2)(d-4)h^3+d(3d-8)ah^2+2da^2h.
\end{aligned}
\]
By \cref{conv:projective-bundle},
\[
  \pi_*(h)=1,
  \qquad
  \pi_*(h^2)=a,
  \qquad
  \pi_*(h^3)=a^2-c_2(S^\vee).
\]
Therefore
\[
  \Theta_3
  =d(d-1)(d-2)a^2-d(d-2)(d-4)c_2(S^\vee).
\]
Equivalently,
\[
  \Theta_3
  =d(d-1)(d-2)\sigma_1^2-d(d-2)(d-4)\sigma_{1,1}.
\]
\end{example}

\begin{example}[The quartic case]\label{ex:quartic}
When $n=3$, one has $d=4$ and $G_3\cong \Gr(2,10)$. The formulas above give
\[
  \Theta_2=12\sigma_1,
\]
\[
  \Theta_3=24\sigma_1^2,
\]
and
\[
\begin{aligned}
  \Theta_4
  &=\pi_*\left[4h(2h+a)(2a)(-2h+3a)\right] \\
  &=24a^3+32a\,c_2(S^\vee).
\end{aligned}
\]
Thus, the universal jet formula specializes to the contact calculations for determinant binary quartics.
\end{example}

\section{Collision strata of binary forms}

The divisor $\cD_n$ is the pullback of the discriminant hypersurface in the space of binary forms. The same principle applies to all determinant-root collision strata.

Let $U$ be a two-dimensional vector space. For a partition
\[
  \lambda=(\lambda_1,\dots,\lambda_r)\vdash d,
\]
let
\[
  B_\lambda\subseteq \Sym^dU^\vee
\]
be the closed cone obtained as the closure of the locus of binary forms of the shape
\[
  F=\prod_{j=1}^r \ell_j^{\lambda_j},
\]
where the $\ell_j\in U^\vee$ are pairwise nonproportional. This is the binary-form collision stratum of type $\lambda$.

The cone $B_\lambda$ is invariant under $\GL(U)$. Hence, it determines an equivariant Chow class
\[
  [B_\lambda]_{\GL_2}=P_\lambda(c_1,c_2)
  \in A^*_{\GL_2}(\operatorname{pt})=\bbZ[c_1,c_2].
\]
Equivariant Chow theory is used here in the sense of Edidin--Graham \cite{EG98}.

Now, form the vector bundle
\[
  \cE_d:=\Sym^dS^\vee\otimes M
\]
on $G_n$. The universal determinant section is a section
\[
  \Delta:G_n\to \cE_d.
\]
The collision cone $B_\lambda$ then globalizes to a relative closed cone
\[
  \cB_\lambda(S,M)\subset \cE_d.
\]
We define the determinant-collision locus
\[
  \cD_\lambda
  :=
  \Delta^{-1}\bigl(\cB_\lambda(S,M)\bigr).
\]
Set-theoretically, $\cD_\lambda$ is the locus of pencils whose determinant binary form lies in the collision stratum $B_\lambda$.

\begin{proposition}\label{prop:collision}
The refined Chow class of the determinant-collision locus $\cD_\lambda$ is
\[
  [\cD_\lambda]^{\mathrm{ref}}
  =
  P_\lambda\bigl(c_1(S^\vee),c_2(S^\vee)\bigr).
\]
If $\cD_\lambda$ has the expected codimension and no excess components, then this refined class is the ordinary fundamental class of $\cD_\lambda$.
\end{proposition}

\begin{proof}
Since $\cB_\lambda(S,M)\subset \cE_d$ is the associated relative cone to the $\GL_2$-invariant cone $B_\lambda\subset \Sym^dU^\vee$, its class in the Chow ring of $\cE_d$ is obtained from the equivariant class $P_\lambda(c_1,c_2)$ by substituting
\[
  c_i=c_i(S^\vee).
\]
The factor $M$ is constant, and contributes no Chern classes. Pulling this class back by the section $\Delta:G_n\to \cE_d$ gives the refined inverse image class
\[
  \Delta^![\cB_\lambda(S,M)]
  =P_\lambda\bigl(c_1(S^\vee),c_2(S^\vee)\bigr).
\]
This is precisely $[\cD_\lambda]^{\mathrm{ref}}$. Under the expected-codimension hypothesis, refined pullback agrees with the ordinary fundamental class.
\end{proof}

\begin{example}[The discriminant partition]\label{ex:disc-partition}
For
\[
  \lambda=(2,1^{d-2}),
\]
the collision stratum is the discriminant hypersurface in $\Sym^dU^\vee$. Its equivariant class is
\[
  P_\lambda(c_1,c_2)=d(d-1)c_1.
\]
Therefore \cref{prop:collision} gives
\[
  [\cD_\lambda]=d(d-1)c_1(S^\vee),
\]
which is \cref{thm:main}.
\end{example}

\section{The squarefree locus and orbit parameters}

Let
\[
  \cU_n:=G_n\setminus \cD_n
\]
be the squarefree locus. A pencil $L\in \cU_n$ has exactly $d=n+1$ singular quadrics, all with multiplicity one.

\begin{proposition}\label{prop:squarefree-orbits}
On a dense open subset of $\cU_n$, the $\PGL(V)$-orbit of a pencil is determined by the unordered configuration of the $d$ roots of its determinant binary form on $\bbP^1$, modulo the natural action of $\PGL_2$ on the parameter line. Consequently, the coarse orbit space of the squarefree locus is birational to
\[
  \Conf_d(\bbP^1)/(\PGL_2\times S_d).
\]
\end{proposition}

\begin{proof}
Let $L\in \cU_n$. Since the determinant binary form is not identically zero, the pencil contains a nonsingular quadric. Choose one and call it $Q_\infty$. Choose another member $Q_0$ of the pencil. The quadratic form $Q_\infty$ identifies $V$ with $V^\vee$, and the operator
\[
  T:=Q_\infty^{-1}Q_0
\]
is self-adjoint with respect to $Q_\infty$. The roots of
\[
  \det(Q_0-\lambda Q_\infty)
\]
are the eigenvalues of $T$. On the squarefree locus, these eigenvalues are distinct. Hence, $T$ is diagonalizable, and its eigenspaces are pairwise orthogonal for $Q_\infty$.

Therefore, there is a basis of $V$ in which
\[
  Q_\infty=\sum_{i=1}^d x_i^2,
  \qquad
  Q_0=\sum_{i=1}^d \lambda_i x_i^2,
\]
with the $\lambda_i$ distinct. Thus, the pencil is projectively equivalent to
\[
  \left\langle
  \sum_{i=1}^d x_i^2,
  \sum_{i=1}^d \lambda_i x_i^2
  \right\rangle.
\]
The unordered set $\{\lambda_1,\dots,\lambda_d\}\subset \bbP^1$ is precisely the root divisor of the determinant binary form.

Changing the basis of the pencil changes the coordinate on the parameter line by $\PGL_2$. Permuting the diagonal basis vectors permutes the roots. Hence, the generic orbit is determined by an unordered configuration of $d$ points of $\bbP^1$, modulo the action of $\PGL_2$.
\end{proof}

The dimension count agrees with this description. Since
\[
  \dim W=\binom{n+2}{2}=\frac{(n+1)(n+2)}2,
\]
we have
\[
  \dim G_n=2\left(\binom{n+2}{2}-2\right)=(n+1)(n+2)-4.
\]
Meanwhile,
\[
  \dim\PGL(V)=(n+1)^2-1.
\]
Therefore,
\[
\begin{aligned}
  \dim G_n-\dim\PGL(V)
  &=(n+1)(n+2)-4-\bigl((n+1)^2-1\bigr) \\
  &=n-2.
\end{aligned}
\]
On the other hand,
\[
  \dim\left(\Conf_{n+1}(\bbP^1)/(\PGL_2\times S_{n+1})\right)
  =(n+1)-3=n-2.
\]

For $n=3$, this quotient is one-dimensional, and it is measured by the classical $j$-invariant of a binary quartic. This is what we used in \cite{Krishna26} to construct the one-parameter family of codimension-one orbit closures in $\Gr(2,10)$.

\section{The case of pencils in \texorpdfstring{$\bbP^3$}{P3}}

Let $n=3$, so $d=4$ and
\[
  G_3=\Gr(2,\Sym^2\bbk^{4,\vee})\cong \Gr(2,10).
\]
The determinant of the universal pencil is a binary quartic. The discriminant divisor is
\[
  \cD_3=\{L:\det(q)|_L\text{ has a multiple root}\}.
\]
By \cref{thm:main},
\[
  [\cD_3]=4\cdot 3\,\sigma_1=12\sigma_1.
\]

This is the boundary divisor class appearing in the study of codimension-one $\PGL_4$-orbit closures in $\Gr(2,10)$. More precisely, the rational $j$-map on the squarefree locus is obtained by assigning to a pencil the $j$-invariant of its determinant binary quartic. The universal discriminant theorem explains why the boundary of this construction has class $12\sigma_1$.

The quartic case also illustrates the distinction between determinant-collision strata and fine Segre strata. The determinant root type $2+1+1$ gives the divisor $\cD_3$. Inside deeper boundary strata, however, one must distinguish whether a singular quadric has rank three, rank two, or rank one. These refinements are invisible to the determinant root multiplicities alone and require separate incidence calculations.

\section*{Acknowledgments}

The author thanks Professor Joe Harris for his guidance.

\bigskip

\noindent\textsc{Department of Mathematics, Harvard University, 1 Oxford Street, Cambridge, MA 02138, USA}\par
\noindent\textit{Email address:} \texttt{akrishna@college.harvard.edu}

\end{document}